\documentclass[a4paper, 11pt]{amsart}
\usepackage{longtable, stackrel, multicol, fancybox, tcolorbox}
\usepackage{bm,amsfonts,amsmath,amssymb,mathrsfs,bbm,amsthm} 
\usepackage{graphicx, color,hyperref,eurosym, eucal,palatino}

\definecolor{darkblue}{rgb}{0.2,0.2,0.6}
\definecolor{darkblue2}{rgb}{0.2,0.2,0.9}
\definecolor{superdarkblue}{rgb}{0.2,0.2,0.3}
\hypersetup{colorlinks=true, linkcolor=darkblue,citecolor=darkblue, 
	urlcolor = superdarkblue}
\definecolor{citegreen}{rgb}{0.2,0.2,0.6}
\usepackage{marginnote}
\usepackage{geometry}
\usepackage{scrextend}
\changefontsizes{10.5pt}

\newcommand\OpM{\sfH_{\bb,\cM}}
\newcommand\OpB{\sfH_{\bb,\cB^\circ}}

\newcommand\frmM{\frh_{\bb,\cM}}
\newcommand\frmB{\frh_{\bb,\cB^\circ}}
\newcommand\DtNM{\sfD_{\lm,\cM}}

\newcommand\EvM{\lm_\bb(\cM)}
\newcommand\EvB{\lm_\bb(\cB^\circ)}

\newcommand\Opm{\sfH_{\bb,\Lm_\frm}}
\newcommand\Opb{\sfH_{\bb,\Lm_\frb}}
\newcommand\frmm{\frh_{\bb,\Lm_\frm}}

\newcommand\sg{\sigma}

\newcommand\Lm{\Lambda}

\newcommand\bb{\beta}

\newcommand\nb{\nabla}
\newcommand\frm{\mathfrak{m}}

\newcommand\frb{\mathfrak{b}}
\newcommand\frh{\mathfrak{h}}
\newcommand\Op{\sfH_{\Lm_\frm}}
\newcommand\Ev{\lm(\Lm_\frm)}

\newcommand{\eg}{{\it e.g.}\,}
\newcommand{\ie}{{\it i.e.}\,}
\newcommand{\cf}{{\it cf.}\,}

\renewcommand\and{\qquad\text{and}\qquad}

\newcommand\sm{\setminus}
\newcommand\tm{\times}
\newcommand\form{\frh_{\Lm_\frm}}

\newcommand{\coloneqq}{:=}

\newcommand{\comm}[1]{}

\renewcommand\aa{\alpha}
\newcommand\lm{\lambda}

\newcommand\p{\partial}

\renewcommand\tt{\theta}

\newcommand\ii{{\mathsf{i}}}

\newcommand\arr{\rightarrow}

\newcommand\s{\rm Spec}
\newcommand\sd{\rm Spec_{\rm disc}}
\newcommand\sess{\rm Spec_{\rm ess}}

\newcommand\distqq{\rho_{\p\cM}}

\newcommand\areaqq{A_\cM}
\newcommand\areadisk{A_{\cB^\circ}}

\DeclareMathOperator{\inj}{inj}

\newcommand\dd{{\mathsf{d}}}

\newcounter{counter_a}
\newenvironment{myenum}{\begin{list}{{\rm(\roman{counter_a})}}%
{\usecounter{counter_a}
\setlength{\itemsep}{1.ex}\setlength{\topsep}{0.8ex}
\setlength{\leftmargin}{5ex}\setlength{\labelwidth}{5ex}}}{\end{list}}

\usepackage[latin1]{inputenc}
\usepackage[T1]{fontenc}

\numberwithin{figure}{section}
\numberwithin{equation}{section}
\theoremstyle{plain}
\newtheorem*{thm*}{Theorem}
\newtheorem{thm}{Theorem}[section]

\newtheorem{lem}[thm]{Lemma}
\newtheorem{prop}[thm]{Proposition}

\newtheorem{cor}[thm]{Corollary}

\theoremstyle{remark}

\theoremstyle{plain}

\def\frb{{\mathfrak b}}

      \def\dC{{\mathbb C}}

      \def\dR{{\mathbb R}}
\def\dS{{\mathbb S}}      
      
   \def\dZ{{\mathbb Z}}

   \def\sfB{{\mathsf B}}   
\def\sfD{{\mathsf D}}      
   \def\sfH{{\mathsf H}}

   \def\cB{{\mathcal B}}   
\def\cD{{\mathcal D}}

\def\cM{{\mathcal M}}   \def\cN{{\mathcal N}}

\newcommand{\dom}{\mathrm{dom}\,}

\makeatletter
\def\section{\@startsection{section}{1}\z@{.9\linespacing\@plus\linespacing}%
	{.7\linespacing} {\fontsize{13}{14}\selectfont\bfseries\centering}}
\def\paragraph{\@startsection{paragraph}{4}%
	\z@{0.3em}{-.5em}%
	{$\bullet$ \ \normalfont\itshape}}

\@addtoreset{equation}{section}
\makeatother

\newcommand{\xx}{{\boldsymbol{x}}}

\newcommand{\eps}{\varepsilon}

\renewcommand{\aa}{\alpha}

\theoremstyle{definition}

\newcommand\soutD{\bgroup\markoverwith
{\textcolor{DarkGreen}{\rule[.5ex]{2pt}{1pt}}}\ULon}

\makeatletter
\newcommand{\vast}{\bBigg@{3}}
\newcommand{\Vast}{\bBigg@{5}}
\makeatother

\begin{document}
%
\title[Spectral isoperimetric inequalities for Robin Laplacians]{Spectral isoperimetric inequalities for Robin Laplacians on 2-manifolds and unbounded cones}

\author{Magda Khalile} 
\address{M. Khalile: Institut f\"ur Analysis, Leibniz Universit\"at Hannover, Welfengarten 1, 30167Hannover, Germany.}
\email{magda.khalile@math.uni-hannover.de}
\author{Vladimir Lotoreichik}
\address{V. Lotoreichik: Department of Theoretical Physics, Nuclear Physics Institute, 
	Czech Academy of Sciences, 25068 \v Re\v z, Czech Republic.}
\email{lotoreichik@ujf.cas.cz}

\begin{abstract}
	We consider the problem of geometric optimization of the lowest eigenvalue for the Laplacian on a compact, 
	simply-connected two-dimensional
		 manifold with boundary 
	subject to an attractive Robin boundary condition. We prove that in the sub-class of 
	manifolds with the Gauss curvature bounded from above by a constant $K_\circ \ge 0$ and under the constraint of fixed perimeter, the geodesic disk of  constant curvature $K_\circ$ maximizes the lowest Robin eigenvalue. In the same geometric setting, it is proved that the spectral isoperimetric inequality holds for the lowest eigenvalue of the Dirichlet-to-Neumann operator. 
Finally, we adapt our methods to Robin Laplacians acting on unbounded three-dimensional cones to show that,  	
	 under a constraint of 
	fixed perimeter of the cross-section, the lowest Robin  eigenvalue is maximized by
	the circular cone. 
\end{abstract}
%

\keywords{Robin Laplacian,
	2-manifold, unbounded conical domain, lowest eigenvalue, spectral isoperimetric inequality, parallel coordinates}
\subjclass[2010]{35P15 (primary); 58J50 (secondary)} 

\maketitle

\section{Introduction}
%
The long history of spectral isoperimetric inequalities was initiated by Lord Rayleigh in his book,  \emph{The theory of sound} \cite{R77}. He conjectured that among membranes of same area, that are fixed along their boundaries, the circular one has the lowest fundamental frequency. This conjecture was proved in any dimension by Faber \cite{F23} and Krahn \cite{K24} and is now known as the so-called \emph{Faber-Krahn inequality}. Later on, the isoperimetric inequality was extended to the case of Neumann boundary condition, and is sometimes referred to as the \emph{reversed} Faber-Krahn inequality as in that case the ball is the maximizer of the first non-zero Neumann eigenvalue \cite{S54, W56}. 
%
%

Spectral optimisation for the Robin Laplacian
is  however quite a new topic of growing interest; see 
the reviews~\cite{BFK17, L19} for a list of existing results and numerous challenging open problems, and the references therein. 
This could be explained as the majority of the classical methods developed for 
Dirichlet and Neumann boundary conditions
either fail or should be modified in the Robin case due to the necessity
to control the boundary term in the Rayleigh quotient.  
Moreover, the results and methods heavily depend on the sign of the parameter in the Robin
boundary condition. To be more precise, for $\cM\subset \mathbb R^d$ being a bounded domain with a sufficiently smooth boundary,  and $\beta \in\mathbb R$, let us consider the spectral problem
\begin{equation}
\label{robin}
\begin{aligned}
-\Delta u &= \lambda u, &\text{ on }& \cM, \\
\partial_\nu u + \beta u &= 0, &\text{ on }& \partial \cM,
\end{aligned}
\end{equation}
where $\partial_\nu  u $ stands for the outward normal derivative of the function $u$. 
%
%
It is proved,
for positive boundary parameters $\beta >0$, that the ball is the unique minimizer of the lowest Robin eigenvalue of \eqref{robin} 
under a fixed volume constraint~\cite{Bossel_1986, BG10, Daners_2006}.  
%
%
For negative parameters $\beta <0$, it was first conjectured in \cite{Bar77} that among the class of smooth bounded domains of same volume the ball should maximize the lowest eigenvalue of \eqref{robin}. However,
the conjecture was disproved in~\cite{FK15} in all space dimensions $d\ge 2$ by showing that the spherical shell yields a smaller lowest Robin eigenvalue than the ball of the same volume provided that the negative Robin parameter is sufficiently large by absolute value. 
This result was unexpected, because the ball is not an optimizer of the lowest eigenvalue of the Laplacian, and the topic attracted since then considerable attention \cite{AFK17,
	BFNT18, FL18, FNT15, FNT16, PP15, V19}. It also gave rise to a new conjecture
 in~\cite{AFK17} stating that in two dimensions,   
 the inequality should hold under area constraint in the class of simply-connected domains.
 %
 %
On the contrary, under a fixed perimeter constraint, some positive results were obtained. In particular, 
 it is proved in~\cite{AFK17} that the disk maximizes the lowest Robin eigenvalue of \eqref{robin}, for any $\beta<0$,  among all planar domains with fixed perimeter. 
 The result has then been extended to any space dimensions $d\geq 2$  in \cite{BFNT18},  among the narrower class of convex domains.
 The isoperimetric optimization problem was also studied for domains in the form $\mathcal D \backslash \overline{\mathcal M}$,  where $\mathcal D$ is either a bounded convex domain in $\mathbb R^d$ \cite{BFNT18}, or the whole Euclidean space ($\cD = \mathbb R^d$)~\cite{KL17a, KL17b}. 
 Finally, it is worth noticing that, while the majority of the above results were obtained in the setting of domains with smooth boundaries, only a few
 results are known for Lipschitz boundaries  \cite{BFNT18, FK18, KW18, L19}.

%
%
In the present paper, we are concerned with the optimization of the lowest Robin eigenvalue of \eqref{robin}, denoted by $\lambda_\beta(\mathcal M)$, when $\beta <0$, under a perimeter constraint.
%
%
We first go beyond the Euclidean setting and generalize the main result of~\cite{AFK17} to two-dimensional Riemannian manifolds (2-manifolds) with boundary. 
This new result can be viewed as the spectral analogue of the classical isoperimetric inequality on $2$-manifolds. To be more specific, consider $\cM$ a $C^\infty$-smooth, compact, simply-connected $2$-manifold with $C^2$-smooth boundary and assume that its Gauss curvature is bounded from above by a constant $K_\circ \ge 0$. Denote by $\cB^\circ$ the geodesic disk of constant curvature $K_\circ$ having the same perimeter as $\cM$. Then, the classical isoperimetric inequality states that, under the additional assumption $\lvert \cM\rvert$, $\lvert \cB^\circ\rvert \leq \frac{2\pi}{K_\circ}$ if $K_\circ >0$, there holds
\[
\lvert \cM\rvert \leq \lvert \cB^\circ\rvert,
\]
where $\lvert \cM\rvert$ and $\lvert \cB^\circ\rvert$ denote, respectively, the areas of $\cM$ and $\cB^\circ$. Under the same assumptions we are able to prove that
\begin{align}
\label{intro_main}
\lambda_\beta(\cM)\leq \lambda_\beta ( \cB^\circ), \quad \text{for all }\beta <0,
\end{align}
see Theorem~\ref{mainthm}.
The analysis of this
optimization problem relies on
abstract functional inequalities on $2$-manifolds, whose proofs employ the method of parallel coordinates~\cite{PW61,Savo_2001} and some geometric corollaries of the isoperimetric inequality~\cite{OS79}.
Related bounds on $\lm_\bb(\cM)$
in the spirit of the Hersh inequality are recently obtained in~\cite{S19}. 
For positive Robin parameters, $\beta >0$, a spectral isoperimetric inequality
for $\EvM$
in any dimension has  been proved very recently in \cite{CGH19} under certain constraints on the curvatures of both the manifold $\cM$ and its boundary $\partial\cM$.
Due to the close connection between the Robin Laplacian  and the Dirichlet-to-Neumann map, we are able to obtain as a consequence an analogue of~\eqref{intro_main} for the lowest eigenvalue of the Dirichlet-to-Neumann map, which is stated in Proposition~\ref{corDtN}. 


%
%

Second, we return back to the Euclidean setting and treat spectral optimization for the Robin Laplacian on
a special class of unbounded three-dimensional domains.
Namely, we consider an unbounded cone
$\Lambda_\frm\subset \mathbb R^3$ with $C^2$-smooth cross-section $\frm \subset \dS^2$, where $\dS^2$ denotes the unit sphere in $\mathbb R^3$. It has been proved, see e.g.~\cite{BP16}, that the Robin Laplacian is well defined on this class of unbounded domains. 
The main motivation to study Robin Laplacians on cones comes from the fact that these operators play an important role in the spectral asymptotic behaviour of Robin Laplacians acting on non-smooth Euclidean domains with conical singularities~\cite{BP16, LP08}. 
In this paper, we are interested in the optimization of the lowest eigenvalue of \eqref{robin}, when $\cM$ is replaced by $\Lambda_\frm$, under a fixed perimeter constraint on the cross-section $\frm$. We are able to prove that the circular cone yields the maximum of the lowest Robin eigenvalue on unbounded cones with simply-connected $C^2$-smooth cross-sections of fixed perimeter, see Theorem~\ref{thm_cone} for a precise statement. The technique employed in the proof strongly relies on the dilation invariance of the cones, which allows us to reduce the study to the level of cross-sections, and again apply  the method of parallel coordinates. 

\section{Main results}\label{sec:results}
%
\subsection{Two-dimensional manifolds}
\label{ssec:result1}
Our first main objective is to generalize the spectral isoperimetric
inequality~\cite[Thm. 2]{AFK17} for the lowest Robin eigenvalue
to a class of two-dimensional Riemannian manifolds with boundary.
Let $(\cM, g)$ be a ($C^\infty$-)smooth compact
two-dimensional, simply-connected  Riemannian manifold 
with $C^2$-smooth boundary $\p\cM$,  equipped with a smooth Riemannian metric $g$.
Such a manifold $\cM$ 
is diffeomorphic to the Euclidean disk and,
in particular, its Euler characteristic is equal to $1$.
We denote by $\dd V$ the area-element on $\cM$ while the arc-element of $\p\cM$ is denoted by $\sigma$.
With a slight abuse of notation, the area of $\cM$ and its perimeter will be respectively denoted by
\[
	\lvert \cM \rvert \coloneqq \int_\cM \dd V
	\and 
	\lvert \p \cM\rvert  \coloneqq \int_{\p \cM} \dd\sigma. 
\]
Let $K \colon \cM \arr\dR$ be the Gauss curvature
on $\cM$ and let the constant $K_\circ \ge 0$ be such that
\begin{equation}\label{eq:K}
	\sup_{x\in\cM} K(x) \le K_\circ.
\end{equation}
In the following we denote by $\cN_\circ$ the two-dimensional Riemannian manifold without boundary of constant non-negative curvature $K_\circ$. Note that, depending on $K_\circ$, the manifold $\cN_\circ$ can be identified with:  the sphere of radius $1/\sqrt{K_\circ}$, if $K_\circ > 0$ and the Euclidean plane, if $K_\circ = 0$.

As usual, $\nabla$ and $-\Delta$
stand, respectively, for the gradient and the positive Laplace-Beltrami operator on $\cM$, which are defined 
through the metric $g$ in local coordinate charts. 
The $L^2$-space $(L^2(\cM), (\cdot,\cdot)_{L^2(\cM)})$ 
and the $L^2$-based first-order Sobolev space $H^1(\cM) = \{u\in L^2(\cM)\colon |\nb u|\in L^2(\cM)\}$ are defined in the standard way. 
For the coupling constant $\beta<0$, we consider the sesquilinear form 
\[
	\frmM [u]
	:=
	\int_\cM \lvert \nabla u \rvert^2 \dd V
	+ 
	\beta \int_{\p\cM} \lvert u \rvert^2 \dd\sigma, 
	\qquad \dom\frmM:= H^1(\cM).
\]
It is semibounded and closed and hence defines a unique self-adjoint operator in $L^2(\cM)$, the Robin Laplacian on $\cM$, denoted by $\OpM$ and acting as 
\begin{align*}
	\OpM u &\coloneqq -\Delta u, \\ 
	\dom \OpM &\coloneqq
	\left\{ u \in H^1(\cM)\colon 
	\Delta u \in L^2(\cM), \p_\nu u|_{\p\cM} + 
	\beta u|_{\p\cM} = 0
	\right\},
\end{align*}
where $\nu$ denotes the outer unit normal to the boundary. It is worth to remark that the Neumann trace
$\p_\nu u|_{\p\cM}$ should be understood in a weak sense; \cf~\cite[Sec. 2]{AM12}.
Thanks to compactness of the embedding $H^1(\cM)\hookrightarrow L^2(\cM)$, the operator $\OpM$ has a compact resolvent and we denote by $\EvM$ its lowest eigenvalue. Applying the min-max principle to $\OpM$, see e.g. \cite[Sec.~XIII.1]{RS4}, its lowest eigenvalue  can be characterized by
\begin{align}\label{minmax_q}
	\EvM
	=
	\inf_{u \in H^1(\cM)\backslash \{0\}} 
	\frac{ \frmM[u] }{ \lVert u \rVert^2_{L^2(\cM)} }.
\end{align}
The variational characterisation~\eqref{minmax_q}
implies that $\EvM <0$ for all $\beta < 0$. 
 
Our main result in this setting is as follows.
\begin{thm}\label{mainthm}
	Let the manifolds $\cM$ and $\cN_\circ$
	 be as above, and let $\cB^\circ \subset \cN_\circ$ be a geodesic disk.
	Assume that $|\p \cM| = |\p\cB^\circ|$. 
	If $K_\circ >0$,  
	assume additionally that $\lvert \cM \rvert, \lvert \cB^\circ \rvert \leq \frac{2\pi}{K_\circ}$. Then there holds
	\[
		\EvM \leq	\EvB,
		\qquad \text{for all}\,\,\bb < 0. 
	\]
\end{thm}
The proof of this result is given in Subsection~\ref{ssec:proof}. The min-max principle allows us to reduce the problem
to a comparison between the Rayleigh quotient of $\OpM$
evaluated on a proper test-function and 
the Rayleigh quotient of $\OpB$ evaluated on its ground-state. The key-step is the construction of the proper test-function which is made by a transplantation  from $\cB^\circ$ to $\cM$ of the ground-state of $\OpB$.
The comparison of the Rayleigh quotients relies on 
functional isoperimetric inequalities on $2$-manifolds
proved in Proposition~\ref{prop:main} using the method of parallel coordinates and the geometric isoperimetric inequality recalled in Subsection~\ref{prel}. The idea of this construction is reminiscent of that in~\cite{AFK17, BFNT18, FK15, PW61}. 

As mentioned above, this result was already known in the Euclidean case, i.e. $K \equiv K_\circ=0$, and proved in the more general setting of $C^2$-smooth domains without the assumption of simply-connectedness, see \cite[Thm. 2]{AFK17}. Within our method, the latter assumption cannot be dropped out as we need, for technical reasons,  to consider manifolds with Euler characteristic being equal to one. However, the necessity of assuming simply-connectedness in the non-Euclidean setting remains an open question. On the contrary,
the additional assumption $\lvert \cM\rvert, \lvert\cB^\circ\rvert\leq \frac{2\pi}{K_\circ}$ when $K_\circ >0$ is necessary. We construct
in Subsection~\ref{ssec:counter}
a counterexample if this assumption is dropped.
In order to characterize
the case of equality in Theorem~\ref{mainthm}, a different
technique is needed. We leave open the question of
whether $\EvM = \EvB$ implies $\cM = \cB^\circ$.

%

\subsection{Unbounded three-dimensional conical domains}\label{ssec:result2}
Our second main objective is to prove a spectral isoperimetric inequality for the lowest discrete Robin eigenvalue on unbounded three-dimensional cones.

Let $\frm \subset \dS^2$ be a simply-connected, $C^2$-smooth domain with $|\frm | < 2\pi$, where $\dS^2$ denotes the unit sphere in $\mathbb R^3$. We equip $\dS^2$ with the canonical Riemannian metric induced by the embedding $\dS^2\hookrightarrow\dR^3$.  
The associated unbounded three-dimensional cone $\Lambda_\frm \subset \dR^3$ and its boundary $\partial \Lambda_\frm$ are defined
in the spherical coordinates 
$(r,\varphi,\tt)\in \dR_+\tm [0,2\pi)\tm [-\frac{\pi}{2},\frac{\pi}{2}]$
as the Cartesian products
\begin{equation}\label{eq:conical}
\Lm_\frm \coloneqq \dR_+\tm\frm\and \partial \Lambda_\frm \coloneqq \mathbb R_+\times \partial \frm. 
\end{equation}	

For the boundary parameter $\bb < 0$, we consider the Robin Laplacian $\Opm$ acting in the Hilbert space 
$L^2(\Lm_\frm)$, and defined as the unique self-adjoint operator in $L^2(\Lm_\frm)$ associated with the closed and  semibounded sesquilinear form 
\begin{equation}
\label{eq:form_cone}
	\frmm [u]
	:= 
	\int_{\Lm_\frm}
	\lvert \nb u \rvert^2 \dd x 
	+
	\bb \int_{\p \Lm_\frm} \lvert u \rvert^2 \dd\sg, \qquad 
	\dom\frh_{\bb,\Lambda_\frm} := H^1(\Lm_\frm),
\end{equation}
where  $\dd\sg$ is the surface measure on $\p\Lm_\frm$, see e.g.~\cite[Lem. 5.2]{BP16}. Notice that, as the domain $\Lambda_\frm$ is unbounded, the essential spectrum of $\Opm$ is no longer empty, and we need, in particular, to add some assumptions on the cross-section $\frm$ to ensure the existence of the discrete spectrum. In the following, assume that
\begin{align}\label{hyp1}
	\lvert \p\frm \rvert < 2 \pi \and |\frm|< 2\pi.
\end{align}
For simple geometric reasons, the second assumption in~\eqref{hyp1} on $\frm$  excludes the possibility for $\dR^3\sm\Lm_\frm$ to be convex.
Hence, we conclude from~\cite[Thm. 1 and Cor. 8]{P16} that
\begin{equation*}
	  \sess(\Opm) = [-\beta^2,\infty)\and \#\sd(\Opm) = \infty.
\end{equation*}
Applying the min-max principle \cite[Sec.~XIII.1]{RS4} to the operator $\Opm$, the lowest eigenvalue  $\lambda_\beta(\Lambda_\frm)$ can be characterized by
\begin{align}\label{lowesteigen}
	\lambda_\beta(\Lambda_\frm)
	= 
	\inf_{u \in H^1(\Lm_\frm)\sm\{0\}}
	\frac{\frh_{\beta,\Lambda_\frm}[u]}{\lVert u \rVert^2_{L^2(\Lm_\frm)}}.
\end{align}

The following result states that the circular cone is a maximizer for the lowest Robin eigenvalue among all cones with fixed perimeter of the cross-section.
\begin{thm}\label{thm_cone}
	Let $\frm\subset\dS^2$ be a $C^2$-smooth, 
	simply-connected domain 
	and $\frb\subset\dS^2$ be a geodesic disk  
	such that $|\p \frm|= | \p \frb| < 2\pi$ and $|\frm|, |\frb| < 2\pi$. Then, one has
	\[
		\lm_\bb(\Lm_\frm) \leq \lm_\bb(\Lm_\frb),\qquad
		\text{for all}\,\,\bb <0.
	\]
\end{thm}

As the lowest eigenvalue of Robin Laplacians acting on circular cones can be explicitly computed, see Lemma~\ref{prop:circular},  we have the following reformulation of Theorem~\ref{thm_cone}.
\begin{cor}

For any $L \in (0,2\pi)$ and  all $\beta <0$, 
\[
\max_{\lvert \partial\frm\rvert= L} \lambda_\beta(\Lambda_\frm)= - \left (\frac{2\pi\bb}{L}\right)^2, 
\]
where the maximum is taken over all
$C^2$-smooth, simply-connected domains  $\frm\subset\dS^2$ satisfying, in addition, $|\frm| < 2\pi$.
\end{cor}
The proof of Theorem~\ref{thm_cone} is given in Subsection~\ref{proof_cone}. 
The idea consists in reducing the problem to the same functional isoperimetric inequalities as in Proposition~\ref{prop:main}, but now on the level of  cross-sections. This reduction rests upon the invariance of the cones under dilations. Contrary to the case of $2$-manifolds,  the strict inequalities $|\p \frm|= | \p \frb| < 2\pi$ and $|\frm|, |\frb| < 2\pi$ are needed in the assumptions of Theorem~\ref{thm_cone}. Indeed, assuming that $| \p \frb| = 2\pi$ and $|\frb| = 2\pi$,  the problem becomes trivial as the associated cone $\Lambda_\frb$ is the half-space and one can see by separation of variables that the  spectrum of $\Opb$ is then purely essential.

This result is reminiscent of the one obtained in~\cite[Thm. 1.3]{EL17} in the setting of 
 $\delta$-interactions supported on conical surfaces.
However, the technique employed here 
 is significantly different from the one used in~\cite{EL17}.  In the latter paper, the spectral problem was first reformulated
as a boundary integral equation by means of the Birman-Schwinger-type principle. Exactly the same technique seems not to be applicable any more for the Robin Laplacian, as the required Green's function is not
explicitly known, to the best of our knowledge.

\section{A spectral isoperimetric inequality on 2-manifolds}\label{manifold}
%
\subsection{Geometric isoperimetric inequalities on 2-manifolds}\label{prel}
%
We stick to the  notation introduced in Subsection~\ref{ssec:result1} and first recall some well-known results.  
For $p,q \in \cM$, we denote by $d(p,q)$ the Riemannian distance between $p$ and $q$, namely the infimum of the lengths of all piecewise smooth curves between $p$ and $q$.  We denote by $\cB_\rho(p)$ the metric disk of radius $\rho >0$ and center $p\in \cM$, 
\[
\cB_\rho(p)\coloneqq \{ x \in \cM: d(p,x) < \rho\}. 
\]
Occasionally, we drop the center and simply write
$\cB_\rho$.
A curve $\gamma \in \cM$ will be called a \emph{geodesic} if it is a minimizing curve between two endpoints with respect to the Riemannian distance. 
For $p\in\cM$ we denote by $\exp_p$ the so-called exponential map defined on a neighborhood of the origin of the tangent space of $\cM$ at $p$ denoted by $T_p\cM$. For a sufficiently small neighborhood of the origin, the exponential map is a diffeomorphism onto its image. The injectivity radius of a point $p\in \cM$ is then defined by
\[
\inj(p)\coloneqq\sup \big\{ \eps > 0\colon {\exp_p}\vert_{B(0,\eps)} \text{ is a diffeomorphism onto its image}\big\},
\]
where $B(0,\eps)\subset T_p\cM$ is the metric disk centered at $0$ and having radius $\eps > 0$. We call \emph{geodesic disk} the image $\exp_p( B(0,\rho))\subset \mathcal M$ for any $\rho < \inj(p)$. It is well known, see e.g. \cite[Cor. 6.11]{Lee97}, that  geodesic disks are in fact metric disks, i.e. the geodesic disk $\exp_p(B(0,\rho))$ is a  metric disk in $\cM$ of center $p$ and radius $\rho$. Recall that in a geodesic disk $\cB_\rho(p)\subset \cM$ there is a unique geodesic between $p$ and any $x\in \cB_\rho(p)$. 

Let us now discuss the classical isoperimetric inequality on two-dimensional manifolds and its corollaries, see \eg the review paper~\cite{OS79} for more details. The proof of the isoperimetric inequality for $2$-manifolds relies, in particular, on the comparison of the area of geodesic disks, respectively, in $\cM$ and $\cN_\circ$, the $2$-manifold with constant curvature $K_\circ \ge 0$. This result will be of interest for us in the next subsection and is stated in the following lemma.
 \begin{lem}\cite[Cor., p.~10]{OS79}\label{comp_disk}\label{lem:OS79}
Let $\cM$ be a smooth, compact, simply-connected $2$-manifold with $C^2$-smooth boundary, having the Gauss curvature $K$ bounded from above by a constant $K_\circ \ge 0$, and let $\cN_\circ$ be a $2$-manifold of constant curvature $K_\circ$. 
If $\cB_\rho\subset \cM$ denotes a geodesic disk of radius $\rho > 0$ and $\cB_\rho^\circ\subset\cN_{\circ}$ the geodesic disk of the same radius, then
\[
	\lvert \cB_\rho\rvert \geq \lvert \cB_\rho^\circ\rvert.
\]
\end{lem}
The isoperimetric inequality on $2$-manifolds is as follows. 
\begin{thm}\label{isopman}\cite[Sec. B]{OS79}
Let $\cM$ be a smooth, compact, simply-connected $2$-manifold with $C^2$-smooth boundary, 
 having the Gauss curvature $K$ bounded from above by a constant $K_\circ \ge 0$. Then, 
the following inequality holds
\[
	\lvert \partial \cM\rvert^2 \geq 
	4\pi \lvert \cM\rvert - K_\circ\lvert \cM \rvert^2,
\]
in which the equality is attained if and only if $K\equiv K_\circ$ and $\cM$ is a geodesic disk.
\end{thm}
The study of the monotony of the quadratic polynomial $f(x)= 4\pi x - K_\circ x^2$ gives the following useful corollaries. 
\begin{cor}\label{corisop1}
Let $\cM$ be as in Theorem~\ref{isopman} and 
$\cB^\circ$ be a geodesic disk of constant curvature $K_\circ \ge 0$. 
	Assume that $\lvert \cM \rvert = \lvert \cB^\circ \rvert$. Then the inequality
	$\lvert \p\cM\rvert \geq \lvert \p\cB^\circ\rvert$
	holds, where the equality is attained if and only if $K\equiv K_\circ$ and $\cM$ is a geodesic disk.
\end{cor}
\begin{cor}\label{corisop2}
Let $\cM$ be as in Theorem~\ref{isopman} and 
$\cB^\circ$ be a geodesic disk of constant curvature $K_\circ \ge 0$. 
	Assume that $\lvert \p \cM \rvert = \lvert \p \cB^\circ \rvert$. If $K_\circ > 0$,
	assume in addition
	 that  $\lvert \cM \rvert, \lvert \cB^\circ \rvert\leq \frac{2\pi}{K_\circ}$.
	Then the inequality $\lvert \cM \rvert \leq \lvert \cB ^\circ \rvert$ holds,
	where the equality is attained if and only if $K\equiv K_\circ$ and $\cM$ is a geodesic disk.
\end{cor}
A geodesic circle on the sphere encloses two geodesic disks. When $K_\circ >0$, the corollary above states that the isoperimetric inequality holds for the ``smaller'' one contained in the hemisphere as $\frac{2\pi}{K_\circ}$ is precisely the area of the hemisphere of radius $1/\sqrt{K_\circ}$. 

%
%

%

%
\subsection{Geometric and analytic properties of parallel coordinates}\label{ssec:parallel}
In what follows, $(\cM,g)$ denotes a smooth compact, simply-connected Riemannian $2$-manifold with $C^2$-smooth boundary. 
 The parameter $K_\circ$ is defined via~\eqref{eq:K}. Recall also that $\cB^\circ$ is a geodesic disk in the manifold $\cN_\circ$ of constant 
Gauss curvature
$K_\circ \geq 0$.
For the sake of brevity, we set $A_\circ := |\cB^\circ|$.
We also assume that the two conditions
below hold:
\begin{myenum}
	\item [\rm (a)] $	L\coloneqq |\p \cM| = |\p\cB^\circ| $. 
	\item [\rm (b)] $\max\{|\cM|, A_\circ\} \le \frac{2\pi}{K_\circ}$, if  $K_\circ > 0$.
\end{myenum}	
According to Corollary~\ref{corisop2} we immediately have
\begin{align}\label{isop1}
	| \cM | \leq A_\circ. 
\end{align}
For a point $x \in \cM$ we introduce $\distqq (x)$, the Riemannian distance between $x$ and the boundary $\p \cM$, \ie 
\[
	\distqq(x)\coloneqq \min_{y\in\p\cM}{d}(x, y).
\]
The function $\distqq$ is Lipschitz continuous and we have $\lvert \nabla \distqq \rvert= 1$ a.e. in $\cM$; \cf~\cite[Sec. 2]{Savo_2001}. The subset ${\rm Cut}(\cM)\subset\cM$, where the function $\distqq$ is not differentiable, has zero measure and is called the \emph{cut-locus}. 
Furthermore, we introduce the in-radius of $\cM$ by
\begin{equation}\label{eq:in_radius}
R_\cM\coloneqq \max_{x\in \cM} \distqq(x).
\end{equation}
The in-radius of $\cM$ is the radius of the largest metric disk that can be inscribed into $\cM$. In what follows, the respective disk in $\cM$ will be denoted by $\cB_{R_\cM}$. We point out that such a disk can be non-unique.  Due to the assumptions made on $\cM$ and $\cB^\circ$ we can compare their in-radii.
\begin{lem}\label{isop2}
Under the assumptions {\rm (a)} and {\rm (b)} there holds
\[
R_\cM\leq R_{\cB^\circ}. 
\]
\end{lem}
\begin{proof}
We will proceed by \emph{reductio ad absurdum} and suppose that $R_\cM > R_{\cB^\circ}$.
	
Let us first consider the case $K_\circ = 0$. The manifold $\cM$ has thus a non-positive Gauss curvature and the Cartan-Hadamard  theorem yields that any two points are connected by a unique geodesic, i.e. 
in that case all the metric disks are geodesic.  In particular $\cB_{R_\cM}\subset \cM$ is a geodesic disk.
Due to Lemma~\ref{comp_disk} we obtain the inequality $\lvert \cB_{R_{\cM}} \rvert \geq A_\circ$, which contradicts \eqref{isop1}. 

If $K_\circ >0$, we need to control locally the injectivity radius of $\cM$ to be able to conclude that $\cB_{R_\cM}$ is a geodesic disk. By \cite[Thm. 2.5.4]{BZ88}, two situations are allowed:
\begin{enumerate}
\item either $R_\cM < \frac{\pi}{\sqrt{K_\circ}}$ and for $x\in \cM$ such that $\distqq(x)=R_\cM$ there holds $\inj(x)= R_\cM$;
\item or $R_\cM \geq \frac{\pi}{\sqrt{K_\circ}}$ and then there exists a point $x\in \cM$ for which we simultaneously have $\distqq(x) \geq \frac{\pi}{\sqrt{K_\circ}}$ and $\inj(x) \geq \frac{\pi}{\sqrt{K_\circ}}$. 
\end{enumerate}
Assume that we are in the situation ($2$). Then, the $2$-manifold $\cM$ contains a geodesic disk of radius $\frac{\pi}{\sqrt{K_\circ}}$ which implies by Lemma~\ref{comp_disk} that 
\[
\lvert \cM\rvert \geq \lvert \cB_{\frac{\pi}{\sqrt{K_\circ}}}\rvert \geq \lvert \cB^\circ_{\frac{\pi}{\sqrt{K_\circ}}}\rvert= \frac{4 \pi}{K_\circ},
\]
where $\cB_{\frac{\pi}{\sqrt{K_\circ}}}\subset \cM$ and $\cB^\circ_{\frac{\pi}{\sqrt{K_\circ}}}\subset \cN_\circ$ are geodesic disks of radius $\frac{\pi}{\sqrt{K_\circ}}$. This contradicts the assumption (b). Hence, only the situation ($1$) is possible, i.e. the largest disk inscribed into $\cM$ is a geodesic disk and we get again a contradiction to inequality~\eqref{isop1} using Lemma~\ref{comp_disk}. 
\end{proof}
%
%
Let us consider the following auxiliary functions:
\begin{equation}\label{eq:LmAm}
\begin{aligned}
	L_\cM&\colon [0, R_\cM] \to \dR_+,& 
	\quad L_\cM(t)&\coloneqq 
	\big|	\{x\in\cM\colon \distqq(x)= t\}\big|, \\
	\areaqq &: [0, R_\cM]\to [0, | \cM| ],& \quad \areaqq(t)&\coloneqq 
	\big| \{x \in\cM\colon \distqq(x) < t\}\big|.
\end{aligned}
\end{equation}
Clearly, $L_\cM(0)= L$ and $\areaqq(R_\cM)= 
|\cM|$. 
The value $\areaqq(t)$ is simply the area of the sub-domain of $\cM$, which consists of the points located at the distance less than $t$ from its boundary $\p\cM$. We can define similarly the same functions associated to the geodesic disk $\cB^\circ$ on the manifold $\cN_{\circ}$ of constant curvature $K_\circ\ge 0$  and  we have the explicit formulae:
\begin{center}
{\small	
\renewcommand{\arraystretch}{1.8}
\begin{tabular}{|c|c|c|}
\hline
 & $K_\circ= 0$ & $K_\circ >0 $\\
 \hline
 $L_{\cB^\circ} (t) =$& $ 
 2\pi (R_{\cB^\circ} - t) $&$ 2\pi \frac{\sin(\sqrt{K_
 	\circ} (R_{\cB^\circ} -t))}{\sqrt{K_\circ}} $
 \\
 \hline
 $\areadisk(t) = $&$ 
 A_\circ -\pi (R_{\cB^\circ} -t) ^2 
 $&$ 
 A_\circ - 2\pi \frac{1- \cos( \sqrt{K_\circ} (R_{\cB^\circ} - t))}{K_\circ}$\\
 \hline
\end{tabular}}
\end{center}
It will also be convenient to introduce the subset
$\cM(t)\subset\cM$ defined by
\begin{equation}\label{eq:mt}
	\cM(t) 
	:=
	\{x \in\cM\colon \distqq(x) > t\},
\end{equation}
and we trivially have $\lvert \cM\rvert= \lvert \cM(t)\rvert + \areaqq(t)$. 
The analytic properties of the functions in~\eqref{eq:LmAm} have been studied in~\cite{bandle,fial,hart,Savo_2001} and we summarize some of them in the proposition below.
\begin{prop}{\cite[App. 1, Prop.~A.1]{Savo_2001}}, \cite[Chap. I, Sec. 3.6]{bandle}\label{prop_savo}
	Let the functions $\areaqq$ and $L_\cM$ be as in~\eqref{eq:LmAm}. Then the following hold:
	\begin{myenum}
		\item $\areaqq$ is Lipschitz continuous
		and increasing; $L_\cM$ is differentiable almost everywhere. 
		\item $\areaqq'(t)= L_\cM(t) > 0$
		for almost every $t\in [0,R_\cM]$.
		\item	$L'_\cM(t) \leq \displaystyle \int_{\cM (t)}  K(x) \dd V(x) - 2\pi$ for almost every $t\in [0,R_\cM]$.
	\end{myenum}
\end{prop}
Note that $|\cM(t)| \le |\cB^\circ(t)|$ by~\eqref{isop1}
for $t = 0$ and this inequality turns out to persist
for $t > 0$. Namely, we have the following statement.
\begin{lem}\label{ineqmb}
	Let the sets $\cM(t)$ and $\cB^\circ(t)$ be defined as in~\eqref{eq:mt}. Then the following inequality 
	$|\cM(t)| \leq |\cB^\circ(t)|$ holds
	for almost every $t\in [0, R_\cM]$. 
\end{lem}	

\begin{proof}
	Define the function $r\colon [0, R_\cM] \to \dR_+$ so that 
	\begin{align}\label{newradius}
		\lvert \cB_{r(t)}^\circ\rvert = \lvert \cM (t) \rvert,
	\end{align}
	where $\cB_{r(t)}^\circ$ is the geodesic disk of radius $r(t) \geq 0$ on the manifold $\cN_{\circ}$ of constant Gauss curvature $K_\circ$. Due to Proposition~\ref{prop_savo}\,(i), the function $r$ is Lipschitz continuous and by the co-area formula we have 
	\[
		\frac{\dd}{\dd t} \lvert \mathfrak \cB_{r(t)}^\circ\rvert = r'(t) \lvert \p \cB_{r(t)}^\circ \rvert,
	\]
	for almost every $t\in [0, R_\cM]$.  Thus, differentiating \eqref{newradius} with the help of Proposition~\ref{prop_savo}\,(ii) one has
	\[
		\lvert \p \cM(t)\rvert 
		= 
		- \frac{\dd}{\dd t} 
		\lvert \cM(t)\rvert = - r'(t) \lvert \p \cB_{r(t)}^\circ\rvert, \quad \text{a.e. in } [0, R_\cM].
	\]
	Corollary~\ref{corisop1} gives
	\[
		r'(t) = -\frac{\lvert \p\cM(t) \rvert}{\lvert \partial \cB_{r(t)}^\circ\rvert} \leq - 1,
	\]
	and integrating the above inequality one obtains
	\begin{align}\label{ineqradius}
		r(t) \leq r(0) - t.
	\end{align}
	By definition~\eqref{newradius} of the radius $r$ and \eqref{isop1} we have $\lvert \cB_{r(0)}^\circ\rvert = \lvert \cM(0)\rvert= \lvert \cM \rvert \leq A_\circ$, which implies, in particular, that $r(0) \leq R_{\cB^\circ}$. Replacing this inequality in~\eqref{ineqradius} we finally get that $r(t) \leq R_{\cB^\circ} - t$, where $R_{\cB^\circ} -t$ is simply the radius of $\cB^\circ(t)$, and this concludes the proof.
\end{proof}
The following lemma will be the main ingredient in the proof of the isoperimetric functional inequalities of Proposition~\ref{prop:main}. 
\begin{lem}\label{ineq4}
For almost every $t\in[0, R_\cM]$ there holds 
\[
	A'_\cM(t) \leq A'_{\cB^\circ} (t).
\]
\end{lem}
\begin{proof}
	The function $L_\cM$ defined in \eqref{eq:LmAm} may not be continuous. Nevertheless, it is proved 
	in~\cite[Cor. 6.1]{hart} that for any $t\in [0, R_\cM]$ there holds 
	\[
		L_\cM(t) - L_\cM(0) \leq \int_0 ^t L'_\cM(s) \dd s.
	\]
	Using Proposition~\ref{prop_savo}\,(iii) this gives
	\begin{align*}
		L_\cM(t) \leq L + \int_0^t 
		\int_{\cM(s)} K(x) \dd V(x) 
		\dd s -2\pi t.
	\end{align*}
	%
	Applying Lemma~\ref{ineqmb} with the inequality $K \leq K_\circ \in[0,\infty)$ one has
	\begin{align}\label{ineq4_1}
		L_\cM(t) \leq L + K_\circ \int_0^t \lvert \cB^\circ(s) \rvert \dd s - 2\pi t.
	\end{align}
	It is straightforward to notice that when $K_\circ = 0$ the right hand side of the above inequality is equal to $L_{\cB^\circ} (t)$. If $K_\circ >0$, using the explicit formula 
	\[
		\lvert \cB^\circ(s)\rvert = 2 \pi \frac{1- \cos( \sqrt{K_\circ} (R_{\cB^\circ} -s))} {K_\circ},
	\]
	we can compute the right-hand side of \eqref{ineq4_1} and get
	\begin{align*}
		L_\cM (t) & \le L + K_\circ \int_0^t \lvert \cB^\circ(s) \rvert \dd s - 2\pi t\\
		&=
		L + 2\pi \left ( t + \frac{\sin( \sqrt{K_\circ} (R_{\cB^\circ} - t)) -\sin( \sqrt {K_\circ} R_{\cB^\circ} )}{\sqrt{K_0}}\right) -2\pi t = L_{\cB^\circ}(t).
	\end{align*}
	Hence, for any $K_\circ\ge 0$ the inequality \eqref{ineq4_1} becomes
	\[
		L_\cM(t) \leq L_{\cB^\circ}(t), \quad \text{ for a.e. } t \in [0, R_\cM]. 
	\]
	The equality in Proposition~\ref{prop_savo}\,(ii) concludes the proof. 
\end{proof}

Let $\psi \in C^\infty([0,R_{\cB^\circ}])$ be an
arbitrary real-valued function.
Due to the properties of $\areaqq$ given in Proposition~\ref{prop_savo}, there exist functions $\phi_\cM\in C^{0,1}([0, |\cM|])$ and $\phi_{\cB^\circ}\in C^{0,1}([0,A_\circ])$\footnote{$C^{0,1}$ denotes
	the class of Lipschitz continuous functions.}  satisfying
\begin{align}\label{parallel1}
	\psi|_{[0, R_\cM]} = \phi_\cM \circ \areaqq
	\and \psi = \phi_{\cB^\circ} \circ A_{\cB^\circ}. 
\end{align}
Consider the test-functions 
\[
\begin{aligned}
	u_\cM & := \psi\circ\rho_{\p\cM} =  \phi_\cM \circ \areaqq \circ \rho_{\p\cM},\\
	u_{\cB^\circ} & := \psi\circ\rho_{\p\cB^\circ} =  \phi_{\mathcal B^\circ}\circ A_{\cB^\circ} \circ \rho_{\p\cB^\circ}.
\end{aligned}
\]
In the following proposition we show functional
isoperimetric inequalities, which are satisfied
by the above test-functions.
\begin{prop}\label{prop:main}	
	Let $\psi \in C^\infty([0,R_{\cB^\circ}])$ be an
	arbitrary real-valued function and let the functions $u_\cM$ and $u_{\cB^\circ}$ be associated to 
	$\psi$ as above.
	Then $u_\cM \in H^1(\cM)$,	$u_{\cB^\circ} \in H^1({\cB^\circ})$  and the following properties hold,	
\begin{myenum}
	\item $\| u_\cM\|_{L^2(\cM)} \le 
	\| u_{\cB^\circ}\|_{L^2(\cB^\circ)}$,
	\item $\|\nb u_\cM\|_{L^2(\cM;{\dC^2})} \le 
	\|\nb u_{\cB^\circ}\|_{L^2(\cB^\circ;{\dC^2})}$,
	\item $\|u_\cM|_{\p\cM}\|_{L^2(\p\cM)} = 
	\|u_{\cB^\circ}|_{\p\cB^\circ}\|_{L^2(\p\cB^\circ)}$.
\end{myenum}	
\end{prop}	
\begin{proof}
	Lipschitz continuity of $\phi_\cM,\phi_{\cB^\circ}, \rho_{\p\cM}, \rho_{\p\cB^\circ}$ and 
	Proposition~\ref{prop_savo}\,(i) 
	imply that  $u_\cM \in H^1(\cM)$ and $u_{\cB^\circ}\in H^1(\cB^\circ)$.
	Employing the parallel coordinates together with the co-area formula, see~\cite[Eq. 30]{Savo_2001}
	for more details, we have
	\begin{align*}
	\|\nb u_\cM \|^2_{L^2(\cM;{\dC^2})}
	 =
	\int_0^{R_\cM} | \phi_\cM'(\areaqq(t))|^2 
	(\areaqq '(t))^3 \dd t.
	\end{align*}
	Using~\eqref{parallel1}, and applying further Lemma~\ref{isop2}, and Lemma~\ref{ineq4}, one obtains
	\begin{equation}\label{parallel2}
	\begin{aligned}
	\|\nb u_\cM \|^2_{L^2(\cM;{\dC^2})}
	& =
	\int_0^{R_\cM} | \psi'(t)|^2 \areaqq '(t) \dd t\\
	& 
	\le
	\int_0^{R_{\cB^\circ}} | \psi'(t)|^2 A_{\cB^\circ} '(t) \dd t\\
	& =
	\int_0^{R_{\cB^\circ}} | \phi'_{\cB^\circ}(A_{\cB^\circ}(t))|^2 (A_{\cB^\circ} '(t))^3 \dd t
	=
	\|\nb u_{\cB^\circ} \|^2_{L^2(\cB^\circ;{\dC^2})}.
	\end{aligned}
	\end{equation}
	Following the same steps we also get
	\begin{equation}\label{parallel3}
	\begin{aligned}
	\| u_\cM \|^2_{L^2(\cM)} 
	& =
	\int_0^{R_\cM} | \phi_\cM(\areaqq(t))|^2 
	\areaqq '(t) \dd t\\
	& =
	\int_0^{R_\cM} | \psi(t)|^2 \areaqq '(t) \dd t
	\le
	\int_0^{R_{\cB^\circ}} | \psi(t)|^2 A_{\cB^\circ} '(t) \dd t\\
	& =
	\int_0^{R_{\cB^\circ}} | \phi_{\cB^\circ}(A_{\cB^\circ}(t))|^2 A_{\cB^\circ}'(t) \dd t
	=
	\| u_{\cB^\circ}\|^2_{L^2(\cB^\circ)}.
	\end{aligned}
	\end{equation}
	Let us focus on the traces of $u_\cM$ and $u_{\cB^\circ}$. It is easy to see that for any $p\in\p \cM$
	and any $q\in \p\cB^\circ$ we have $u_\cM(p) = u_{\cB^\circ}(q) = \psi(0)$. Hence, we obtain
	\begin{align}\label{parallel4}
	\| u_\cM|_{\p\cM}\|^2_{L^2(\p \cM)} 
	=
	\| u_{\cB^\circ}|_{\p\cB^\circ} \|^2_{L^2(\p\cB^\circ)}
	=
	L | \psi(0)|^2. &\qedhere
	\end{align}
\end{proof}	
\subsection{Proof of Theorem~\ref{mainthm}}
\label{ssec:proof}
The proof relies on the min-max principle and consists in comparing the Rayleigh quotients of $\OpM$ and $\OpB$ with the aid of Proposition~\ref{prop:main}. 
For a geodesic disk $\cB^\circ$ of constant Gauss curvature $K_\circ$ the first Robin eigenvalue $\EvB$ is simple and also strictly negative, provided that $\beta <0$. Moreover, an associated eigenfunction, denoted by $u_{\cB^\circ}$, is  $C^\infty$-smooth, by standard elliptic regularity,  and radial with respect to the geodesic polar coordinates, with the pole being the center of $\cB^\circ$; \cf Appendix~\ref{app:ground-state}. 
%
Hence, there exists a real-valued $\psi \in C^\infty ([0, R_{\cB^\circ}])$ such that $u_{\cB^\circ}= \psi\circ \rho_{\p\cB^\circ}$. 
In the following we use the notation $u_\cM \coloneqq \psi \circ \rho_{\p\cM}$. 
Proposition~\ref{prop:main} implies
$\frmM [u_\cM] \leq \frmB[u_{\cB^\circ}] < 0$
and $\|u_\cM\|_{L^2(\cM)} \le \|u_{\cB^\circ}\|_{L^2(\cB^\circ)}$.
Hence, the min-max principle yields
\[
	\EvM 
	\leq 
	\frac{\frmM [u_\cM] }{\lVert u_\cM \rVert^2_{L^2(\cM)}} \le 
	 \frac{\frmB[u_{\cB^\circ}]}{\lVert u_{\cB^\circ} \rVert^2_{L^2(\cB^\circ)}} = \EvB. 
	 \]

\subsection{A counterexample based on weak-coupling}\label{ssec:counter}

In this subsection we show that the
additional assumption $|\cM|,|\cB| \le \frac{2\pi}{K_\circ}$ when $K_\circ > 0$ in Theorem~\ref{mainthm} is necessary. The idea is to find a counter-example on the sphere and to prove that, similarly to the geometric isoperimetric inequality, the spectral isoperimetric inequality holds only for domains contained in the hemisphere. 
To this aim we assume that $\cM  \subset \dS^2$.
In that case $K_\circ = 1$ and $\cN_{\circ}=\dS^2$. Let
$\cM\subset\dS^2$ be such that $|\cM| > 2\pi$
and $|\p\cM| < 2\pi$. Assume also that $\cM$ is not a geodesic disk. Furthermore, let $\cB^\circ\subset\dS^2$ be the  
geodesic disk with $L\coloneqq|\p\cB^\circ| = |\p\cM|$ and $|\cB^\circ| \le 2\pi$. 
Hence, we obviously have $|\cM| > |\cB^\circ|$. Since $|\p(\dS^2\sm\cM)| = |\p\cM|$ and $|\dS^2\sm\cM| < 2\pi$, Corollary~\ref{corisop2} yields that $|\dS^2\sm\cM| < |\cB^\circ|$ and hence also $|\cM| > |\dS^2\sm\cB^\circ|$. 

Before getting to the conclusion, let us recall some basic properties of the function
$\mathbb R_- \ni \beta \mapsto \lambda_\beta(\mathcal M)$
that can be found \eg in \cite{BFK17}. The eigenvalue $\lambda_\beta(\cM)$ is simple and analytic in the parameter $\beta$, and thus one can compute in the standard way its first derivative at $\bb = 0$,
\begin{equation}\label{eq:derivative_beta_0}
	\left(\frac{\dd}{\dd \beta} \lambda_\beta(\cM)\right)\bigg\vert_{\beta= 0}= \frac{\lvert \partial \cM\rvert }{\lvert \cM\rvert},
\end{equation}
\cf~\cite[Lem. 2.11]{AFK13} where the computation is made for $\beta >0$
in the Euclidean setting  and can be easily adapted to $\beta <0$ and lifted up to manifolds.
Using the formula~\eqref{eq:derivative_beta_0} we get
\[
\begin{aligned}
	\lim_{\bb\arr 0^-}\frac{\EvM - \EvB}{|\bb|} & = L\left(\frac{1}{|\cB^\circ|} - \frac{1}{|\cM|}\right) >0,\\	
	\lim_{\bb\arr 0^-}\frac{\EvM - \lm_\bb(\dS^2\sm\cB^\circ)}{|\bb|} & =  L\left(\frac{1}{|\dS^2\sm\cB^\circ|}- \frac{1}{|\cM|} \right) > 0.
\end{aligned}
\]
Hence, for  $\bb < 0$ with sufficiently small
absolute value, the inequalities
$\EvM > \EvB$ and $\EvM > \lm_\beta(\dS^2\sm\cB^\circ)$,
opposite to the one in Theorem~\ref{mainthm},
hold.

\subsection{An application to the Dirichlet-to-Neumann map}\label{proof_DtN} 
In this subsection, we prove a counterpart of the inequality in Theorem~\ref{mainthm} for the
lowest eigenvalue  of the Dirichlet-to-Neumann map. For $\lambda <0$, the Dirchlet-to-Neumann map, denoted by $\sfD_{\lambda,\cM}$, is defined in $L^2(\partial \cM)$ by 
\begin{align*}
\sfD_{\lambda, \cM} \varphi&= \partial_\nu u| _{\partial \cM}, \\
\dom \sfD_{\lambda,\cM}&\coloneqq \Big \{ \varphi \in L^2(\partial \cM): \exists u \in H^1(\cM) \text{ such that } -\Delta u = \lambda u \text{ with } u = \varphi \text{ on } \partial \cM \\
& \text{ and } \partial_\nu u|_{\partial \cM} \text{ exists in } L^2(\partial \cM)\Big\}, 
\end{align*}
see e.g. \cite{AM12}. 
Here again, the equation $-\Delta u = \lambda u$ and the Neumann trace $\partial_\nu u| _{\partial \cM}$ should be understood in the weak sense. It has been proved in~e.g. \cite[Sec. 2]{AM12} that the operator $\sfD_{\lambda,\cM}$ is self-adjoint in $L^2(\partial \cM)$, non-negative, and has a compact resolvent. We denote by $\sigma_\lm(\cM)$ its lowest eigenvalue. This eigenvalue is sometimes referred to as the lowest Steklov-type eigenvalue~\cite{FNT15}. Remark that $\sigma_\lambda (\mathcal M) >0$ when $\lambda <0$ while the first Steklov eigenvalue satisfies $\sigma_0(\mathcal M)= 0$, and the optimization problem is trivial in that case. 

The next lemma contains a  relation between the spectra of the Robin Laplacian and the Dirichlet-to-Neumann map and provides two of its consequences. 
\begin{lem}\label{lem:AM}\cite[Thm. 3.1, Prop. 3]{AM12}
Let $\cM$ be a smooth, compact, simply-connected $2$-manifold with $C^2$-smooth boundary. Then, the following properties hold:	
	\begin{myenum}
	\item $-\sigma\in\s(\DtNM)$ if and only if
		$\lm\in\s(\sfH_{\sigma,\cM})$, for all $\sigma < 0$.
		\item $\dR_-\ni \sigma \mapsto\lm_\sigma(\cM)$
		is strictly increasing.
		\item $\dR_-\ni \lm \mapsto\sigma_\lm(\cM)$
		is strictly decreasing.
\end{myenum}	
\end{lem}
The next result is a direct consequence of the previous lemma and gives the link between $\lambda_\sigma(\cM)$ and $\sigma_\lambda(\cM)$. 
\begin{cor}\label{cor:AM}
	For any $\lm <0$ and $\sigma < 0$, there holds 	
	$\sigma_\lm(\cM) = -\sigma$
	if and only if $\lm_\sigma(\cM) = \lm$.
\end{cor}		
\begin{proof}
	Assume that $\sigma_\lm(\cM) = -\sigma$, then
	by Lemma~\ref{lem:AM}\,(i) we have	$\lm\in\s(\sfH_{\sigma,\cM})$
	and thus	$\lm_\sigma(\cM) \le \lm$. Suppose
	for a moment that 
	$\lm' := \lm_\sigma(\cM) < \lm$. Then, one has 
	$-\sigma\in\s(\sfD_{\lm',\cM})$ and hence
	$\sigma_{\lm'}(\cM) \le \sigma_\lambda(\cM)$. It contradicts the monotonicity 
	of $\lambda\mapsto \sigma_\lambda(\cM)$ stated in Lemma~\ref{lem:AM}\,(iii).
	
	Assume now that $\lm_\sigma(\cM) = \lm$, then
	by Lemma~\ref{lem:AM}\,(i) we have	$-\sigma\in\s(\sfD_{\lm,\cM})$
	and thus	$\sigma_\lm(\cM) \le -\sigma$. Suppose
	for a moment that 
	$\sigma' := -\sigma_\lm(\cM) > \sigma$. Then, one has $\lm\in\s(\sfH_{\sigma',\cM})$ and
	$\lm_{\sigma'}(\cM) \le \lambda_\sigma(\cM)$. It contradicts the monotonicity of $\sigma\mapsto \lambda_\sigma(\cM)$ stated in
	 Lemma~\ref{lem:AM}\,(ii).
\end{proof}	
The isoperimetric inequality for the Dirichlet-to-Neumann map on $2$-manifolds is as follows. 
\begin{prop}\label{corDtN}
	Let $\cM$ be a smooth, compact, simply-connected $2$-manifold with $C^2$-smooth boundary and Gauss curvature bounded by above by a constant $K_\circ \ge 0$, and $\mathcal N_\circ$ be the 2-manifold of constant curvature $K_\circ$. Let 
	$\cB^\circ\subset\cN_{\circ}$ be a geodesic disk.
	Assume that $|\p \cM| = |\p\cB^\circ|$. 
	If $K_\circ >0$,
	assume  additionally that $\lvert \cM \rvert, \lvert \cB^\circ \rvert \leq \frac{2\pi}{K_\circ}$. Then there holds
	\[
		\sigma_\lm(\cM) \leq \sigma_\lm(\cB^\circ) ,
		\qquad \text{for all}\,\,\lm < 0. 
	\]
\end{prop}
\begin{proof}
	Let $\lm < 0$ be fixed and set 
	$\sigma := -\sigma_\lm(\cB^\circ)$,
	$\sigma' := -\sigma_\lm(\cM)$. Corollary~\ref{cor:AM} yields  $\lm_\sigma(\cB^\circ) = \lm_{\sigma'}(\cM)  = \lm$. Applying Theorem~\ref{mainthm}, we get
	\[
		\lm_\sigma(\cM) \le \lm_\sigma(\cB^\circ).
	\] 
	Hence, we conclude from Lemma~\ref{lem:AM}\,(ii) that
	$\sigma \leq \sigma'$.
\end{proof}	


\section{The Robin Laplacian on cones}\label{cone}
%
\subsection{Ground-state of Robin Laplacians on circular cones}\label{def_cone}
%
In what follows we stick to the notation introduced in Subsection~\ref{ssec:result2}. Recall that the cross-section $\frm\subset \dS^2$ satisfies the assumptions \eqref{hyp1} and that the associated cone $\Lambda_\frm\subset \mathbb R^3$ is defined in \eqref{eq:conical}. Upon having identified $\dS^2$ and the product $[0,2\pi)\tm [-\frac{\pi}{2},\frac{\pi}{2}]$, the domain $\Lm_\frm$
can alternatively be characterized
in the Cartesian coordinates $\xx\coloneqq (x_1,x_2,x_3)\in\dR^3$ as
\begin{equation*}\label{key}
	\Lm_\frm = \big\{ (r\cos\tt\cos\varphi, r\cos\tt\sin\varphi, r\sin\tt)\in\dR^3\colon r\in \dR_+,
	(\varphi,\tt)\in \frm\big\}.
\end{equation*}
The point $\xx_{\rm np}$ with the spherical coordinate
$(1,0,\frac{\pi}{2})$ is called the \emph{north pole}
of $\dS^2$.

By invariance of $\Lm_\frm$ under isotropic scaling 
one can show, using the change of variables 
$\xx\mapsto \bb\xx$ in the form \eqref{eq:form_cone}, that the operator $\Opm$ is unitarily equivalent to $\bb^2\sfH_{-1,\Lm_\frm}$ for all $\bb <0$.
Hence, it suffices to work out the case $\bb = -1$
and in what follows
we introduce the shorthand notation 
\begin{equation}\label{inv_dil}
	\Op := \sfH_{-1,\Lm_\frm},\qquad \form := \frh_{-1,\Lm_\frm}\and \lm(\Lm_\frm) :=
	\lm_{-1}(\Lm_\frm).
\end{equation}

Let $\frb\subset \dS^2$ be a geodesic
disk on $\dS^2$ (\ie spherical cap) such that $L\coloneqq | \p \frb|=\lvert \p \frm\rvert < 2\pi$ and $|\frb|  < 2\pi$. For the sake
of definiteness we always assume that $\frb$ is centered
at the north pole $\xx_{\rm np}\in\dS^2$.
 As shown in \cite[Sec. 5]{LP08}, the first eigenpair of $\sfH_{\Lm_\frb}$ can be computed explicitly. 
In order to keep the presentation self-contained, let us recall the proof of this simple result. We first need some basic properties of the one-dimensional self-adjoint Robin Laplacian defined in $L^2(\dR_+)$ as 
\begin{equation}\label{eq:op_B}
	\sfB_\bb f= - f'',\quad \dom\sfB_\bb\coloneqq
	\{ f\in H^2(\dR_+)\colon {f'(0) =\bb f(0)}\},\qquad
	\bb \in\dR.
\end{equation}
\begin{lem}\cite[Ex. 2.4]{LP08}
	Let the self-adjoint operator $\sfB_\bb$ be as in~\eqref{eq:op_B} with $\bb <0$.
	Then  
	\begin{myenum}
		\item $\sess(\sfB_\beta)=[0, \infty)$.
		\item $\sfB_\bb$ has a unique simple eigenvalue  ${\lm_1}(\sfB_\bb)\coloneqq-\beta^2$ and, in particular,	
		\begin{align}\label{rob1D}
			\int_{\dR_+} | f'(t)|^2 \dd t +\bb | f(0)|^2 
			\ge
			-\bb^2 \int_{\dR_+} | f(t)|^2 \dd t, \qquad \forall\, f \in H^1(\dR_+). 
		\end{align}
	\end{myenum}
\end{lem}
The following proposition provides a characterisation for the ground-state of $\sfH_{\Lm_\frb}$.
\begin{prop}\label{prop:circular}
Let $\frb\subset \dS^2$ be a geodesic disk of perimeter $L\coloneqq \lvert \partial \frb\rvert < 2\pi$ and centered at the north pole $\xx_{\rm np}\in \dS^2$. Then, the first eigenvalue of $\Op$ is given by $\lambda(\Lambda_\frb)= - \left (\displaystyle \frac{2\pi}{L}\right) ^2$ and 
\begin{align}
\label{eigenpaircone}
u_\frb(\xx)\coloneqq \exp\left(-\frac{2\pi x_3}{L}\right)
\end{align}	
is an associated eigenfunction. 
\end{prop}
\begin{proof}
Denote by $(r,\tt,z)\in\dR_+\tm \dS^1\tm \dR$ the cylindrical coordinates in $\dR^3$ and by $\aa \in (0,\frac{\pi}{2})$ the geodesic radius of $\frb$. Then, one can write for any $u\in H^1(\Lm_\frb)$,
	\[
		\frh_{\Lm_\frb}[u] 
		=
		\int_0^\infty 
		\int_0^{2\pi} 
		\left[
		\int_{r\cot\aa}^{\infty} 
		\left ( |\p_r u \rvert^2 + \frac{|\p_\tt u|^2}{r^2} + |\p_z u |^2\right )  \dd z 
		-	\frac{| u  (r, \tt, r\cot\aa) |^2}{\sin \aa}\right]r\dd\tt\dd r.
	\]	
	Using the one-dimensional inequality  \eqref{rob1D} we have
	\[
		\int_{r\cot\aa}^{\infty} 
		|\p_z u|^2\dd z - 
		\frac{| u(r,\tt ,r\cot\aa)|^2}{\sin\aa} \geq -\frac{1}{\sin^2\aa} \int_{r\cot\aa}^{\infty} |u | ^2 \dd z,
	\]
	and plugging it into the expression for $\frh_{\Lm_\frb}[u]$ above, we get
	\[
		\frh_{\Lm_\frb}[u]
		\geq 
		-\frac{1}{\sin^2\aa}\|u\|^2_{L^2(\Lm_\frb)}.
	\]
	On the other hand, evaluating the 
	form $\frh_{\Lm_\frb}$ on the $H^1$-function 
	\[
		u_\frb(r,\tt,z) = \exp\left(-\frac{z}{\sin\aa}\right),
	\]	 
	we immediately obtain $\frh_{\Lm_\frb}[u_\frb]= -\frac{1}{\sin^2\aa}  \| u \|^2_{L^2(\Lm_\frb)}$. In order to conclude the proof it remains to notice that $\sin\aa = \frac{L}{2\pi}$. 
\end{proof}

%

%

%


\subsection{Proof of Theorem~\ref{thm_cone}}\label{proof_cone}
%
In the following we denote by $\dd\sigma_\frm$, $\dd\sigma_\frb$ the $2$-dimensional surface measures
on $\frm$ and $\frb$ and by $\dd\lm_{\frm}$, $\dd\lm_{\frb}$ the respective $1$-dimensional
measures on their boundaries $\p\frm$ and $\p\frb$. We make use of the function $\rho_{\p\frm}$ introduced in Subsection~\ref{ssec:parallel}, which denotes the distance to the boundary $\partial \frm$ while $R_\frm$ defined in~\eqref{eq:in_radius} is the in-radius of $\frm$.  

Let $x\coloneqq ( r, \varphi,\theta)\in \frb$ be a point in the geodisic disk $\frb$ written in spherical coordinates. With our convention one has $\theta= \frac{\pi}{2} - R_\frb +\rho_{\p\frb}(x)$. As the eigenfunction $u_\frb$ defined in \eqref{eigenpaircone} does not depend on the $\varphi$-variable in spherical coordinates, 
 there exists, for any fixed $r>0$, a function $\psi_r\in C^{0,1}([0, R_\frb])$ satisfying 
\[
	u_\frb(r, x)= \psi_r(\rho_{\p\frb}(x)), \quad (r,x)\in \mathbb R_+\times \frb.
\]	
As a test function we use 
\[
	u_\frm(r,x)
	:= 
	\psi_r(\rho_{\p\frm}(x)),
	\qquad 
	(r,x)\in \dR_+ \tm \frm.
\]
Applying the inequalities in Proposition~\ref{prop:main}\,(i) and (ii), 
with $\cM = \frm$ and $\cB^\circ = \frb$,  slice-wise for each fixed $r > 0$, we get 
\begin{equation}\label{cl1}
\begin{aligned}
	\| \nb u_\frm\|^2_{L^2(\Lm_\frm;{\dC^3})} 
	&=
	\int_0^\infty \int_\frm |\nb u_\frm|^2 r^2\dd\sigma_\frm\dd r\\
	& = 
	\int_0^\infty 
	\int_\frm
	\left(|\p_r u_\frm|^2 + \frac{|\nb_{\dS^2} u_\frm|^2}{r^2}\right) r^2\dd\sigma_\frm\dd r\\
	& 
	\le
	\int_0^\infty 
	\int_\frb
	\left(|\p_r u_\frb|^2 + \frac{|\nb_{\dS^2}  u_\frb|^2}{r^2}\right) r^2\dd\sigma_\frb\dd r
	= \| \nb u_\frb \|^2_{L^2(\Lm_\frb;{\dC^3})},
\end{aligned}
\end{equation}
where $\nb$ is the gradient on $\dR^3$ and $\nabla_{\dS^2}$ is the gradient on $\dS^2$. 
Applying the inequality in Proposition~\ref{prop:main}~(i) in the same
manner we get
\begin{equation}\label{cl2}
\begin{aligned}
	\|u_\frm\|^2_{L^2(\Lm_\frm)} 
	& =
	\int_0^\infty \int_\frm |u_\frm|^2 r^2\dd\sigma_\frm\dd r\\
	& \le
	\int_0^\infty	\int_\frb| u_\frb|^2 r^2\dd\sigma_\frb\dd r
	= 
	\| u_\frb\|^2_{L^2(\Lm_\frb)}.
\end{aligned}
\end{equation}
Employing Proposition~\ref{prop:main}\,(iii) we find
\begin{equation}\label{cl3}
\begin{aligned}
	\| u_\frm|_{\p\Lm_\frm}\|^2_{L^2(\p\Lm_\frm)}
	& = 
	\int_0^\infty \int_{\p\frm} 
	|u_\frm|_{\Lm_\frm}|^2 r\, \dd\lm_\frm \dd r\\
	& =
	\int_0^\infty \int_{\p\frb} |u_\frb|_{\Lm_\frb}|^2 r\, \dd\lm_\frb \dd r
	=
	\| u_\frb|_{\p\Lm_\frb}\|^2_{L^2(\p\Lm_\frb)}.
\end{aligned}
\end{equation}
Note that the inequality~\eqref{cl1} and the equality~\eqref{cl3} yield
\begin{align}
\label{cl4}
	\form[u_\frm] \leq \frh_{\Lm_\frb}[u_\frb] < 0.
\end{align}
Using the variational characterization of $\Ev$ in~\eqref{lowesteigen} and the inequalities~\eqref{cl2},~\eqref{cl4} one obtains
\[
	\Ev\leq \frac{\frh_{\Lm_\frm}[u_\frm]}{\|u_\frm\|^2_{L^2(\Lm_\frm)}} \leq
	\frac{\frh_{\Lm_\frb}[u_\frb]}{\|u_\frb\|^2_{L^2(\Lm_\frb)}} = \lm(\Lm_\frb).
\]

\begin{appendix}
\section{Ground-state of a geodesic disk of constant curvature}\label{app:ground-state}
Let $\cB^\circ$ be a geodesic disk of constant curvature $K_\circ\in[0,\infty)$. The geodesic polar coordinates with the pole at the center of $\cB^\circ$ are the analogue of the usual Euclidean polar coordinates when $K_\circ >0$. Let $p\in\cB^\circ$ be  written as $(r,\tt)$ in the geodesic polar coordinates. The coordinate $r$ measures the length of the unique geodesic $\gamma$ connecting the center of $\cB^\circ$ and $p$ while $\theta$ corresponds to the unique geodesic passing through $p$ and orthogonal to $\gamma$, see \cite[Sec. 4-3]{Str88} for a rigorous definition. The next lemma summarizes some properties of the geodesic polar coordinates. 
\begin{lem}\cite[Section 4-3]{Str88}
On a $2$-manifold with constant curvature $K_\circ\ge 0$, the metric $g_{{\rm pol}}$ takes the form
\[
	g_{{\rm pol}}= \dd r^2 + h(r) \dd\theta^2,
\]
where the function $h$ has the form
\begin{myenum}
\item $h(r)\coloneqq r^2$, if $K_\circ=0$,
\item $h(r)\coloneqq \frac{\sin^2( \sqrt{K_\circ} r)}{ K_\circ}$, if $K_\circ >0$.
\end{myenum}
\end{lem}
We are now able to show that the first Robin
eigenvalue on a geodesic disk of constant curvature is simple and that the corresponding eigenfunction is radially symmetric. 
Let $\rho>0$ and $\cB^\circ_\rho$ be a geodesic disk of constant curvature $K_\circ$ and radius $\rho$. We introduce the
following $L^2$-space on $\cB^\circ_\rho$,
\[
	L^2_{{\rm pol}}(\cB_\rho^\circ)
	\coloneqq 
	L^2\big ( (0, \rho)\times [0, 2\pi); h^{\frac12}(r)\dd r \dd\theta\big),
\]
and also the associated first-order Sobolev space 
\[
	H^1_{{\rm pol}}(\cB_\rho^\circ)
	\coloneqq
	\big\{ v \in L^2_{{\rm pol}}(\cB^\circ_\rho)\colon |\nabla_{{\rm pol}} v| \in L^2_{{\rm pol}}(\cB_\rho^\circ)\big\},
\]
where $\nabla_{\text{pol}}$ stands for the gradient in the geodesic polar coordinates.
 By change of variables, one can write the Robin form on $\cB^\circ_\rho$ in geodesic polar coordinates as
\[
	\frh^{\text{pol}}_{\beta,\cB^\circ_\rho} [v]
	:= 
	\int_0^{2\pi} \int_0^\rho 
	\left ( \lvert \partial_r v \rvert^2 + \frac{\lvert \partial_\theta v \rvert^2}{ h(r)}  \right) h^\frac12(r)\dd r \dd\theta + \beta h^\frac12(\rho) \int_0^{2\pi} \lvert v(\rho,\theta)\rvert^2  \dd\theta,
\]
with $\dom \frh^{\text{pol}}_{\beta, \cB^\circ_\rho}= H^1_{\text{pol}}(\cB_\rho^\circ)$.
 Let us define the following complete family of orthogonal projections in the Hilbert space $L^2_{\rm pol}(\cB_\rho^\circ)$ \emph{via}
\[
	(\Pi_m v)(r,\tt) := 
	\frac{e^{\ii m\tt}}{\sqrt{2\pi}}\int_{0}^{2\pi}
	v(r,\phi) e^{-\ii m\phi}\dd \phi,\qquad m\in\dZ.
\]
Upon an obvious identification, this family induces the orthogonal decomposition of the Hilbert space 
$L^2_{\rm pol}(\cB_\rho^\circ)$,
\[
	L^2_{\rm pol}(\cB_\rho^\circ) \simeq
	\bigoplus_{m\in\dZ} L^2((0,\rho);h^\frac12(r)\dd r),
\]
and the respective 
decomposition of the operator $\sfH_{\beta,\cB_\rho^\circ}$, 
\[
	\sfH_{\beta,\cB_\rho^\circ}
	\simeq \bigoplus_{m\in\dZ}
	\sfH_{\beta,\cB_\rho^\circ}^{[m]},
\]
where the operator $\sfH^{[m]}_{\beta,\cB_\rho^\circ}$, $m\in\dZ$, represents the quadratic form
\[
\begin{aligned}
	\frh_{\beta,\cB_\rho^\circ}^{[m]}[\psi]
	& := 
	\int_0^\rho 
	\left ( \lvert  \psi'(r) \rvert^2 + \frac{m^2\lvert  \psi(r) \rvert^2}{ h(r)}  \right) h^\frac12(r)\dd r + \beta h^\frac12(\rho) \lvert|\psi(\rho)\rvert^2,\\
	\dom\frh_{\beta,\cB_\rho^\circ}^{[m]} & := \left\{\psi\colon \psi,\psi', m\psi h^{-\frac12}\in L^2((0,\rho);h^\frac12(r)\dd r)\right\}.
\end{aligned}
\]
It is clear that $\frh_{\beta,\cB_\rho^\circ}^{[0]}$
is the smallest among the above forms in the sense of the ordering. Moreover, one can prove by standard arguments that 
the ground-state
of $\sfH_{\beta,\cB_\rho^\circ}^{[0]}$
is simple. Furthermore,
we obtain from the min-max principle
that
the lowest eigenvalue $\sfH_{\beta,\cB_\rho^\circ}^{[0]}$
is strictly less than that of
$\sfH_{\beta,\cB_\rho^\circ}^{[m]}$ for any $m\ne 0$.
We conclude that the ground-state of $\sfH_{\bb,\cB_\rho^\circ}$ corresponds to
the ground-state of the self-adjoint fiber operator
 $\sfH_{\beta,\cB_\rho^\circ}^{[0]}$ 
and thus is simple and radially symmetric.

\end{appendix}

\subsection*{Acknowledgments}
M.K. was supported in part by the PHC Amadeus 37853TB funded by the French Ministry of Foreign Affairs and the French Ministry of Higher Education, Research and Innovation.	
V.L. acknowledges
the support by the grant No.~17-01706S of
the Czech Science Foundation (GA\v{C}R)
and by a public grant as part of the Investissement d'avenir
project, reference ANR-11-LABX-0056-LMH, LabEx LMH. 
The authors are deeply indebted to 
Daguang Chen for pointing out a flaw 
in the initial version of the proof for Lemma~\ref{ineq4}.
The authors are also very grateful to
Mikhail Karpukhin, David Krej\v{c}i\v{r}\'{i}k, Zhiqin Lu  and
Konstantin Pankrashkin for fruitful discussions.

%

%

\newcommand{\etalchar}[1]{$^{#1}$}

\end{document}